\documentclass[submission]{FPSAC2024}

\usepackage{enumerate, mathtools}
\usepackage{xcolor}
\usepackage{tikz}
\usepackage{tikz-cd}
\usepackage{subcaption}
\usepackage{url}
\usepackage{stmaryrd}
\usepackage{bbm}
\usepackage{appendix}
\usepackage{bm}
\usepackage{longtable}
\usepackage{tikz-qtree}
\usepackage{mathrsfs}
\usepackage{cleveref}

\usetikzlibrary{arrows}

\numberwithin{equation}{section}
\numberwithin{figure}{section}

\newtheorem{thmabc}{Theorem}

\newtheorem{corabc}[thmabc]{Corollary}
\newtheorem{thm}{Theorem}[section]
\newtheorem*{thm*}{Theorem}
\newtheorem*{con*}{Conjecture}

\newtheorem{cor}[thm]{Corollary}

\crefname{thm}{Theorem}{Theorems}
\crefname{lem}{Lemma}{Lemmas}
\crefname{prop}{Proposition}{Propositions}
\crefname{equation}{Equation}{Equations}

\theoremstyle{definition}

\newtheorem{ex}[thm]{Example}

\DeclareMathOperator{\HS}{HS}

\DeclareMathOperator{\Mat}{Mat}

\DeclareMathOperator{\GL}{GL}

\DeclareMathOperator{\diag}{diag}

\DeclareMathOperator{\GSp}{GSp}

\newcommand{\bfx}{\mathbf{x}}

\newcommand{\bfX}{\bm{X}}
\newcommand{\bfY}{\bm{Y}}

\newcommand{\N}{\mathbb{N}}

\newcommand{\Q}{\mathbb{Q}}
\newcommand{\R}{\mathbb{R}}

\newcommand{\Z}{\mathbb{Z}}
\newcommand{\C}{\mathbb{C}}
\renewcommand{\phi}{\varphi}

\newcommand{\pbinom}[3][]{
    \genfrac{[}{]}{0pt}{}{#2}{#3}_{\ifthenelse{\isempty{#1}}{p}{#1}}
}
\renewcommand{\leq}{\leqslant}
\renewcommand{\geq}{\geqslant}

\renewcommand{\epsilon}{\varepsilon}

\def \mcH {\mathcal{H}}

\def \Fp {\mathbb{F}_p}

\def \Z {\mathbb{Z}}
\def \Zp {\mathbb{Z}_p}
\def \Q {\mathbb{Q}}
\def \N {\mathbb{N}}

\newcommand{\Smithpart}[1]{\bm{\nu}(#1)}
\newcommand{\Smithinc}[1]{\bm{\mu}(#1)}
\newcommand{\Hermite}[1]{\bm{\delta}(#1)}
\newcommand{\prim}{\mathrm{pr}}
\newcommand{\Satakegen}{R}
\newcommand{\HSsp}{\overline{\HS}}

\allowdisplaybreaks

\title{Ehrhart polynomials, Hecke series, and affine buildings}

\author[C. Alfes, J. Maglione, C. Voll]{Claudia
  Alfes\thanks{\href{mailto:alfes@math.uni-bielefeld.de}{alfes@math.uni-bielefeld.de}. Alfes
    was supported by the Daimler and Benz Foundation.}\addressmark{1}, \and Joshua
  Maglione\thanks{\href{mailto:joshua.maglione@universityofgalway.ie}{joshua.maglione@universityofgalway.ie}. Maglione was supported by the DFG-GRK 2297.}\addressmark{2}, \and Christopher
  Voll\thanks{\href{mailto:C.Voll.98@cantab.net}{C.Voll.98@cantab.net}. Alfes
    and Voll are funded by the Deutsche Forschungsgemeinschaft (DFG, German Research Foundation) — SFB-TRR 358/1 2023 — 491392403.}\addressmark{1}}

\address{
  \addressmark{1}Fakult\"at f\"ur Mathematik, Universit\"at Bielefeld, Germany \\ 
  \addressmark{2}School of Mathematical and Statistical Sciences, University of Galway, Ireland
}

\received{\today}

\abstract{ Given a lattice polytope $P$ and a prime $p$, we define a
  function from the set of primitive symplectic $p$-adic lattices to
  the rationals that extracts the $\ell$th coefficient of the Ehrhart
  polynomial of $P$ relative to the given lattice. Inspired by work of
  Gunnells and Rodriguez Villegas in type $\mathsf{A}$, we show that
  these functions are eigenfunctions of a suitably defined action of
  the spherical symplectic Hecke algebra. Although they depend
  significantly on the polytope $P$, their eigenvalues are independent
  of $P$ and expressed as polynomials in $p$. We define local zeta
  functions that enumerate the values of these Hecke eigenfunctions on
  the vertices of the affine Bruhat--Tits buildings associated with
  $p$-adic symplectic groups. We compute these zeta functions by
  enumerating $p$-adic lattices by their elementary divisors and,
  simultaneously, one Hermite parameter. We report on a general
  functional equation satisfied by these local zeta functions,
  confirming a conjecture of Vankov.
}

\keywords{Ehrhart polynomials, Hecke series, affine buildings, Satake
  isomorphism, symplectic lattices}

\usepackage[backend=bibtex]{biblatex}
\begin{document}

\maketitle

\section{Introduction}
\label{sec:intro}

Let $P$ be a fixed full-dimensional lattice polytope in $\R^n$, i.e.\ the convex hull
of finitely many points $V(P)$ in $\Lambda_0 = \Z^n$. Given a lattice $\Lambda$
such that $\Lambda_0\subseteq \Lambda \subseteq \Q^n$, we denote the
\emph{Ehrhart polynomial of $P$ with respect to $\Lambda$} by 
\begin{align}\label{eqn:Ehrhart-def}
  E^\Lambda(P) &= \sum_{\ell=0}^nc_\ell^\Lambda(P) T^n\in\Q[T].
\end{align} 
It is of interest to describe the variation of the coefficients
$c_\ell^\Lambda(P)$ with $\Lambda$ as compared to $c_{\ell}(P)=
c_{\ell}^{\Lambda_0}(P)$; write $E(P)$ for $E^{\Lambda_0}(P)$. For
$g\in \GL_n(\Q)\cap \Mat_n(\Z)$ we define
\begin{align*} 
  g\cdot P = \mathrm{conv}\{g\cdot v ~|~ v\in V(P)\} ,
\end{align*} 
which is again a lattice polytope. We write $\Lambda_g$ for the lattice
generated by the rows of $g\in\GL_n(\Q)$. Thus, for every $g\in \GL_n(\Q)\cap
\Mat_n(\Z)$ , we have
\begin{align}\label{eqn:Ehrhart-id}
  E(g\cdot P) &= E^{\Lambda_{g^{-1}}}(P).
\end{align}
We note that $\Lambda_g \subseteq \Z^n \subseteq \Lambda_{g^{-1}} \subseteq
\Q^n$ for $g\in \GL_n(\Q)\cap \Mat_n(\Z)$ with $|\det(g)| > 1$.  

Gunnells and Rodriguez Villegas~\cite{GRV/07} consider how the
coefficients of $E^{\Lambda}(P)$ from \Cref{eqn:Ehrhart-def} relate to
$E(P)$ for lattices $\Lambda$ such that $\Lambda_0\subseteq \Lambda
\subseteq p^{-1}\Lambda_0 \subseteq \Q^n$. In \Cref{sec:type-A} we
revisit these results from our perspective. In addition, we consider a
symplectic analogue of the work of Gunnells and Rodriguez Villegas.

\subsection{Zeta functions of Ehrhart coefficients}
\label{sec:intro-zeta}
For a prime $p$, we write $\Z_p$ for the ring of $p$-adic integers and
$\mathbb{Q}_p$ for its field of fractions.  Below we define, for each
$n\in\N=\{1,2,\dots\}$ and $\ell\in [2n]_0 = \{0, \dots, 2n\}$, local
zeta functions which we call \emph{Ehrhart--Hecke zeta
functions}. These functions are Dirichlet series in a complex variable
$s$ encoding the ratio of $\ell$th coefficients of the Ehrhart
polynomial of $P$, as the lattice $\Lambda$ varies among symplectic
lattices in~$\Q_p^{2n}$.

Recall the group scheme $\mathrm{GSp}_{2n}$ of symplectic
similitudes. For a ring $K$ its $K$-rational points are, with
$J=\left( \begin{smallmatrix} 0 & I_n \\ -I_n & 0 \end{smallmatrix}
\right)$,
\begin{align*}
  \mathrm{GSp}_{2n}(K) &= \left\{ A \in \GL_{2n}(K) ~\middle|~ A J A^{\mathrm{t}} = \mu(A)J,\ \text{for some } \mu(A)\in K^\times \right\} .
\end{align*}


We set $G_n = \mathrm{GSp}_{2n}(\Q_p)$, $\Gamma_n =
\mathrm{GSp}_{2n}(\Z_p)$, and $G_n^+ = \mathrm{GSp}_{2n}(\Q_p)\cap
\Mat_{2n}(\Z_p)$. The set $G^+_n/\Gamma_n$ is in bijection with the
set of special vertices of the affine building associated with the
group $\GSp_{2n}(\Q_p)$, which is of type $\widetilde{C}_n$.

We define the (\emph{local}) \emph{Ehrhart--Hecke
zeta function} (of type $\mathsf{C}$) as
\begin{align*} 
  \mathcal{Z}^{\mathsf{C}}_{n,\ell,p}(s) &= \sum_{g\in G^+_n/\Gamma_n} \dfrac{c_{\ell}^{\Lambda_{g^{-1}}}(P)}{c_{\ell}(P)} |\Lambda_{g^{-1}} : \Z_p^n|^{-s}. 
\end{align*}
Informally speaking, the zeta function
$\mathcal{Z}^{\mathsf{C}}_{n,\ell,p}(s)$ hence encodes the average
$\ell$th coefficient of the Ehrhart polynomial of $P$ across certain
symplectic lattices.

\subsection{Symplectic Hecke series}
\label{sec:Hecke-series}

The zeta functions of \Cref{sec:intro-zeta} are closely connected to
formal power series over the Hecke algebra associated with the pair
$(G_n^+, \Gamma_n)$. To explain this connection, we establish
additional notation. For $m\in\N$ we define
\begin{align*}
  D_n^{\mathsf{C}}(m) &= \{A \in G_n^{+} ~|~ AJA^{\mathrm{t}} = mJ \}.
\end{align*}
Let $\mcH^{\mathsf{C}}_p = \mcH^{\mathsf{C}}(G_n^+, \Gamma_n)$ be the
spherical Hecke algebra. The Hecke operator $T_n^{\mathsf{C}}(m)$ is
\begin{align*}
  T_n^{\mathsf{C}}(m) &= \sum_{g\in \Gamma_n \setminus D_n^{\mathsf{C}}(m) / \Gamma_n} \Gamma_ng\Gamma_n. 
\end{align*}
The (formal) symplectic Hecke series is defined as
\begin{align}\label{eqn:Hecke-series}
  \sum_{\alpha\geq 0} T_n^{\mathsf{C}}(p^{\alpha})X^{\alpha} \in \mathcal{H}_p^{\mathsf{C}}\llbracket X \rrbracket.
\end{align}
Shimura's conjecture \cite{Shimura/63} that the series in
\eqref{eqn:Hecke-series} is a rational function in $X$ was proved by
Andrianov~\cite{Andrianov/69}. Explicit formulae, however, seem only
to be known for $n \leq 4$; see \cite{Vankov/11}.

We consider the image of the Hecke series in \eqref{eqn:Hecke-series}
under the Satake isomorphism $\Omega : \mcH^{\mathsf{C}}_p \to
\C[x_0^{\pm 1}, \dots, x_n^{\pm 1}]^W$ mapping onto the ring of
invariants of $W$, the Weyl group of $G_n$. For variables $\bm{x} =
(x_0,\dots, x_n)$, we define the (\emph{local}) \emph{Satake
generating function} as
\begin{align*} 
  \Satakegen_{n,p}(\bm{x}, X) &= \sum_{\alpha\geq 0}
  \Omega(T_n^{\mathsf{C}}(p^{\alpha}))X^{\alpha} \in \C[
    \bfx^{\pm1}]\llbracket X \rrbracket.
\end{align*}
and the (\emph{local}) \emph{primitive local Satake generating
function} as
\begin{align}\label{eqn:prim-Satake} 
  \Satakegen_{n,p}^{\prim}(\bm{x}, X) &= \left(1 - x_0X\right)\left(1
  - x_0x_1\cdots x_nX\right)\Satakegen_{n,p}(\bm{x}, X).
\end{align} 

We write $V(\mathscr{X}_n)$ for the set of vertices of
$\mathscr{X}_n$, the affine building $\mathscr{X}_n$ of type
$\widetilde{\mathsf{A}}_{n-1}$ associated with the group
$\mathrm{GL}_{n}(\Q_p)$, viz.\ homothety classes of full lattices
in~$\mathrm{GL}_{n}(\Q_p)$.  In~\cite[Section~3.3]{Andrianov/87}
Andrianov shows, in essence, that $\Satakegen_{n,p}^\prim$ can be
interpreted as a sum over $V(\mathscr{X})$; see
\Cref{thm:Andrianov-formula} below.

For a lattice $\Lambda\leq \Z_p^n$, set $\Smithpart{\Lambda} = (\nu_1
\leq \cdots \leq \nu_n)\in\N_0^n$ if $\Z_p^n/\Lambda \cong
\Z/p^{\nu_1} \oplus \cdots \oplus \Z/p^{\nu_n}$. Setting $\nu_0=0$, we
define
$$
  \bm{\mu}(\Lambda) = (\mu_1,\dots,\mu_n) = (\nu_n - \nu_{n-1},\ \dots,\ \nu_1 - \nu_0).
$$

 Having chosen a $\Z_p$-basis of $\Zp^n$ we associate to each lattice
 $\Lambda\leq \Z_p^n$ a unique matrix
\begin{align}\label{eqn:HermiteNF}
  M_{\Lambda} &= \begin{pmatrix} 
    p^{\delta_1} & m_{12} & \cdots & m_{1n} \\
    & p^{\delta_2} & \cdots & m_{2n} \\
    & & \ddots & \vdots \\
    & & & p^{\delta_n} 
  \end{pmatrix} \in \Mat_n(\Z_p),
\end{align} 
whose rows generate $\Lambda$ and with $0\leq v_p(m_{ij}) \leq
\delta_j$ for all $1\leq i < j \leq n$. The matrix $M_{\Lambda}$ in
\eqref{eqn:HermiteNF} is said to be in Hermite normal form. We set
$\Hermite{\Lambda} = (\delta_1,\dots, \delta_n)$. Clearly each
homothety class $[\Lambda]$ contains a unique representative
$\Lambda_{\mathrm{m}}\leq \Z_p^{n}$ such that
$p^{-1}\Lambda_{\mathrm{m}} \not\leq \Z_p^{n}$.

\begin{thm}[Andrianov]\label{thm:Andrianov-formula}
  Let $n\in\N$, $\bm{a} = (1, 2, \dots, n)\in \N^n$, $\bm{d} = (n,
  n-1, \dots, 1)\in \N^n$, and let $\langle, \rangle$ be the usual dot
  product. Then
 \begin{align*}
    \Satakegen_{n,p}^{\prim}(\bm{x}, X) &= \sum_{[\Lambda]\in
      V(\mathscr{X}_n)} p^{\langle \bm{d}, \Smithpart{\Lambda_{\mathrm{m}}}\rangle
      - \langle \bm{a}, \Hermite{\Lambda_{\mathrm{m}}}\rangle}
    x_1^{\delta_1(\Lambda_{\mathrm{m}})} \cdots
    x_n^{\delta_n(\Lambda_{\mathrm{m}})}(x_0X)^{\nu_{n}(\Lambda_{\mathrm{m}})}.
  \end{align*} 
\end{thm}

\subsection{The Hermite--Smith generating function}
\label{sec:Hermite-Smith}

We define a generating function enumerating finite-index sublattices
of $\Zp^n$ simultaneously by their Hermite and Smith normal forms.
For $n\in\N$, let $\bm{X}=(X_1, \dots, X_n)$ and $\bm{Y} = (Y_1,
\dots, Y_n)$ be variables. The \emph{Hermite--Smith generating
function} is
\begin{align}
  \HS_{n,p}(\bm{X}, \bm{Y}) &= \sum_{\Lambda\leq \Z_p^n}
  \bm{X}^{\Smithinc{\Lambda}}\bm{Y}^{\Hermite{\Lambda}} =
  \sum_{\Lambda\leq \Z_p^n}
  \prod_{i=1}^{n}X_i^{\mu_i(\Lambda)}Y_i^{\delta_i(\Lambda)} \in \Z
  \llbracket \bfX,\bfY \rrbracket.
\end{align}
Clearly, if $\Lambda\leq\Z_p^n$ has finite index, then so does
$p^m\Lambda$ for all $m\in \N_0$. This allows us to extract a
``homothety factor'' from the sum defining $\HS_{n,p}(\bm{X},
\bm{Y})$. The \emph{primitive Hermite--Smith generating function} is
\begin{align}
  \HS_{n,p}^{\prim}(\bm{X}, \bm{Y}) &= \sum_{[\Lambda]\in V(\mathscr{X}_n)} \bm{X}^{\Smithinc{\Lambda_{\mathrm{m}}}}\bm{Y}^{\Hermite{\Lambda_{\mathrm{m}}}} = (1 - X_nY_1\cdots Y_n)\HS_{n,p}(\bm{X}, \bm{Y}) . 
\end{align}

With this generating function we may obtain the primitive local Satake
generating function of \Cref{sec:Hecke-series}, as follows. We define
a ring homomorphism
\begin{align}
  \Phi : \Q\llbracket X_1, X_2, \dots, Y_1, Y_2, \dots\rrbracket & \longrightarrow \Q\llbracket x_0, x_1, \dots, X\rrbracket\nonumber \\
  X_i &\longmapsto p^{\binom{i+1}{2}}x_0X,\nonumber \\
  Y_i &\longmapsto p^{-i}x_i \label{eqn:Phi}
\end{align}
for all $i\in \N_0$. By design of $\Phi$ and
\Cref{thm:Andrianov-formula} we have $\Phi(\HS_{n,p}^{\prim}) =
\Satakegen_{n,p}^{\prim}$.

\begin{ex}
  For $n=2$, the Hermite--Smith generating function is
  \begin{align*}
    \HS_{2, p}(\bm{X},\bm{Y}) &= \dfrac{1 - X_1^2Y_1Y_2}{(1 -
      X_1Y_1)(1 - pX_1Y_2)(1 - X_2Y_1Y_2)},\\
    \Satakegen_{2,p}(\bm{x}, X) &= \dfrac{1 - p^{-1}x_0^2x_1x_2X^2}{(1
      - x_0X)(1 - x_0x_1X) (1 - x_0x_2X) (1 - x_0x_1x_2X)} .
  \end{align*}
\end{ex}

\section{Main results}
\label{sec:main-results}

Interpreting the $\ell$-th coefficients of the Ehrhart polynomial of
the polytope $P$ as a function on a set of (homothety classes of)
$p$-adic lattices invites the definition of an action of the spherical
Hecke algebra $\mcH^{\mathsf{C}}_p$. The latter is generated by a set
of $n+1$ generators $T_n^{\mathsf{C}}(p,0)$,
$T_n^{\mathsf{C}}(p^2,1),\dots, T_n^{\mathsf{C}}(p^2,n)$. It suffices
to explain how these generators act. For $k\in [n]$, define diagonal
matrices in $G_n^+$ as follows:
\begin{align*} 
  D_0 &= \mathrm{diag}(\underbrace{1, \dots, 1}_{n}, \underbrace{p, \dots, p}_n), & D_k &= \mathrm{diag}(\underbrace{1, \dots, 1}_{n-k}, \underbrace{p, \dots, p}_k, \underbrace{p^2, \dots, p^2}_{n-k}, \underbrace{p, \dots, p}_k).
\end{align*}
Set $\mathscr{D}_{n,k}^{\mathsf{C}} = \Gamma_n D_k\Gamma_n/\Gamma_n$.
The set $\mathscr{D}_{n,k}^{\mathsf{C}}$ can be interpreted as the set
of symplectic lattices with symplectic elementary divisors equal to
those of $D_k$. We define
\begin{align*} 
  T_n^{\mathsf{C}}(p,0) E(P) &= \sum_{g\in \mathscr{D}_{n,0}^{\mathsf{C}}} E(g\cdot P), & T_n^{\mathsf{C}}(p^2, k) E(P) &= \sum_{g\in \mathscr{D}_{n,k}^{\mathsf{C}}} E(g\cdot P). 
\end{align*}
For $\ell\geq \N_0$, we define functions
$$
  \mathscr{E}_{n,p,\ell,P} : G_n^+/\Gamma_n \to \C, \quad \Gamma_ng \mapsto c_{\ell}(E^{\Lambda_{g^{-1}}}(P)).
$$ 
Lastly, for all $T\in\mcH^{\mathsf{C}}_p$ set 
\[ 
  T\mathscr{E}_{n,p,\ell,P}(\Gamma_ng) = c_{\ell}(TE^{\Lambda_{g^{-1}}}(P)).
\]  

Recall that $P$ is full-dimensional; for $k\in [n]$, and $\ell\in
[2n]_0$, we define
\begin{align*} 
\nu^{\mathsf{C}}_{n, 0, \ell}(p) &= \dfrac{c_{\ell}(T_n^{\mathsf{C}}(p, 0)E(P))}{c_{\ell}(E(P))}, &
  \nu^{\mathsf{C}}_{n, k, \ell}(p) &= \dfrac{c_{\ell}(T_n^{\mathsf{C}}(p^2, k)E(P))}{c_{\ell}(E(P))}.
\end{align*} 
The notation suggests that the value $\nu_{n,k,\ell}^{\mathsf{C}}(p)$
is independent of the polytope $P$, which is justified by
\Cref{thmmain:Ehrhart-Satake}. General properties of the Ehrhart
polynomial imply that
\begin{align*} 
  \nu^{\mathsf{C}}_{n, n, \ell}(p) &= p^{\ell}, & \nu^{\mathsf{C}}_{n,k,0}(p) &= \#\mathscr{D}_{n,k}^{\mathsf{C}}.
\end{align*}

Every $\Q$-linear homomorphism $\lambda : \mcH^{\mathsf{C}}_p \to \C$ is
uniquely determined by parameters $(a_0,\dots, a_n)\in \C^{n+1}$ such that if
$\psi : \C[x_0^{\pm 1},\dots, x_n^{\pm 1}]\to \C$ is given by $x_i=a_i$ then
$\lambda = \psi\circ\Omega$; see \cite[Proposition 3.3.36]{Andrianov/87}.

\begin{thmabc}\label{thmmain:Ehrhart-Satake}
  The functions $\mathscr{E}_{n,p,\ell, P}$ are Hecke eigenfunctions
  under the action defined above; specifically, for all $k\in [n]$, we
  have
  \begin{align*}
    T_n^{\mathsf{C}}(p, 0)\mathscr{E}_{n,p,\ell, P} &= \nu_{n,0,\ell}^{\mathsf{C}}(p)\mathscr{E}_{n,p,\ell, P}, & T_n^{\mathsf{C}}(p^2, k)\mathscr{E}_{n,p,\ell, P} &= \nu_{n,k,\ell}^{\mathsf{C}}(p)\mathscr{E}_{n,p,\ell, P},
  \end{align*}
  where the $\nu_{n,k,\ell}^{\mathsf{C}}(p)$ are polynomials in $p$
  with integer coefficients which are independent of~$P$. Moreover,
  the parameters associated to $\nu_{n,k,\ell}^{\mathsf{C}}(p)$ are
  $(p^{\ell}, p, p^2, \dots, p^{n-1}, p^{n-\ell})$.
\end{thmabc}

\Cref{tab:nu-values} lists the values of $\nu_{n,k,\ell}^{\mathsf{C}}(p)$ for
small values of $n$ and $k$. 

\begin{table}[h]
  $$
  \begin{array}{l|l|l} \ell & \nu^{\mathsf{C}}_{2,0,\ell}(p) & \nu^{\mathsf{C}}_{2,1,\ell}(p)\\
    \hline
      4 & p^5 + p^4 + p^3 + p^2& p^8 + p^7 + p^6 + p^5 \\
      3 & p^4 + p^3 + p^3 + p^2& 2p^6 + p^5 + 2p^4 -p^3 \\
      2 & p^3 + p^3 + p^2 + p^2& p^5 + 3p^4+p^3-p^2 \\
      1 & p^3 + p^2 + p^2 + p^1& 2p^4 + p^3 + 2p^2 -p \\
      0 & p^3 + p^2 + p+ 1& p^4 + p^3 + p^2 + p \\
  \end{array}
  $$
  \caption{The polynomials $\nu_{2,k,\ell}^{\mathsf{C}}(p)$ for $k\in
  \{0,1\}$ and $\ell\in [4]_0$.}
  \label{tab:nu-values}
\end{table}

\Cref{thmmain:Ehrhart-Satake} enables us to relate
$\mathcal{Z}_{n,\ell,p}^{\mathsf{C}}(s)$ to $\Satakegen_{n,p}(\bm{x}, X)$. Let
$\psi_{n,\ell}$ be the ring homomorphism from $\C\llbracket x_0, x_1, \dots,
X\rrbracket \to \C[t]$ given by 
\begin{align*}
  X &\mapsto t^n & x_0 &\mapsto p^{\ell}, & x_n &\mapsto p^{n-\ell}, & x_i \mapsto p^i.
\end{align*}

\begin{corabc}\label{cormain:Z-G}
  For $n\in \N$ and $\ell\in [2n]_0$ we have, writing $t = p^{-s}$,
  \begin{align*} 
    (\psi_{n,\ell}\circ \Phi)(\HS_{n,p}^{\prim}(\bm{X}, \bm{Y})) =
    \psi_{n,\ell}(\Satakegen_{n,p}^{\prim}) =
        {\mathcal{Z}_{n,\ell,p}^{\mathsf{C}}(s)}{\left(1 -
          p^{\ell-s}\right)\left(1 - p^{\binom{n+1}{2}-s}\right)}.
  \end{align*}
\end{corabc}

Thanks to \Cref{cormain:Z-G}, we can work with $\HS_{n,p}$ to prove
that $\Satakegen_{n,p}$ and $\mathcal{Z}_{n,\ell,p}^{\mathsf{C}}$
satisfy a self-reciprocity property, which proves the conjecture in
\cite[Remark~4]{Vankov/11}.

\begin{thmabc}\label{thmmain:reciprocity}
  Let $n\in\N$. Then $\HS_{n,p}(\bm{X}, \bm{Y})$ is a rational
  function in $\bm{X}$ and $\bm{Y}$. Furthermore, for
  $\bm{X}^{-1}=(X_1^{-1},\dots, X_n^{-1})$ and $\bm{Y}^{-1} =
  (Y_1^{-1},\dots, Y_n^{-1})$, we have
  \begin{align*}
    \left.\mathrm{HS}_{n,p}(\bm{X}^{-1},\bm{Y}^{-1})\right|_{p\to p^{-1}} &= (-1)^np^{\binom{n}{2}}X_nY_1\cdots Y_n\cdot \mathrm{HS}_{n,p}(\bm{X},\bm{Y}).
  \end{align*}
\end{thmabc}

We prove \Cref{thmmain:reciprocity} by writing $\HS_{n,p}$ as a
$p$-adic integral and applying results of \cite{Voll/10}, where the
operation of inverting $p$ is also explained.

\begin{corabc}
  For $n\in\N$ and $\ell\in[2n]_0$, we have
  \begin{alignat*}{2}
    \left. \mathcal{Z}^{\mathsf{C}}_{n,\ell,p}(s)\right|_{p \rightarrow p^{-1}} &=& (-1)^{n+1} p^{n^2+\ell -2ns} \cdot& \mathcal{Z}^{\mathsf{C}}_{n,\ell,p}(s),\\
    \left. R_{n,p}(\bfx, X)\right|_{p \rightarrow p^{-1}} &=& (-1)^{n+1} p^{\binom{n}{2}}x_0^2 x_1\dots x_n X^2 \cdot& R_{n,p}(\bfx, X).
    \end{alignat*}
\end{corabc}

In the next theorem, we determine a formula for the specialization of
$\HS_{n,p}^{\prim}$ which yields $\mathcal{Z}_{n,\ell,p}^{\mathsf{C}}$
by \Cref{cormain:Z-G}. To this end we define
$$\HSsp_{n,p}(\bm{X}, Y) = \HS_{n,p}^{\prim}(\bm{X}, 1, \dots, 1,
Y).$$ We prove that $\HSsp_{n,p}$ is a rational function in the $n+1$
variables $\bm{X}$ and $Y$ and, in addition, the prime~$p$. In order
to describe the formula, we define additional notation. For $I = \{i_1
< \cdots < i_{\ell}\} \subseteq [n-1]$, with $i_{\ell+1}=n$, $k\in
[\ell+1]$, and a variable $Z$, we set
\begin{align*} 
  I^{(k)} &= \{i_j ~|~ j < k\} \cup \{i_j-1 ~|~ j\geq k\}\\ 
  \mathscr{G}_{n,I,k}(Z,\bm{X},Y) &= \left(\prod_{j=1}^{k-1}\dfrac{Z^{i_j(n - i_j - 1)}X_{i_j}}{1 - Z^{i_j(n - i_j - 1)}X_{i_j}}\right) \left(\prod_{j=k}^{\ell} \dfrac{Z^{i_j(n - i_j)}X_{i_j}Y}{1 - Z^{i_j(n - i_j)}X_{i_j}Y}\right). 
\end{align*}

\begin{thmabc}\label{thm:HS.explicit}
  Let $n\in\N$. For $I=\{i_1 < \cdots < i_{\ell}\}_< \subseteq [n-1]$, set 
  \begin{align*} 
    W_{n,I} (Z, \bm{X}, Y) &= \sum_{k=1}^{\ell+1} Z^{-(n-i_{k})}\binom{n-1}{I^{(k)}}_{Z^{-1}}\mathscr{G}_{n,I,k}(Z, \bm{X}, Y) \\
    &\quad + \sum_{k=1}^{\ell} \dfrac{(1 - Z^{-i_j})\mathscr{G}_{n,I,k}(Z, \bm{X}, Y)}{1 - Z^{i_j(n - i_j - 1)}X_{i_j}}\left(\sum_{m=k+1}^{\ell+1}Z^{-(n-i_m)}\right)\binom{n-1}{I^{(k+1)}}_{Z^{-1}} .
  \end{align*}
  Then 
  $$
    \HSsp_{n,p}(\bm{X}, Y) = \sum_{I\subseteq [n-1]} W_{n, I}(p, \bm{X}, Y) \in \Z(p,\bm{X}, Y).
  $$
\end{thmabc}

Via the various substitutions given above, \Cref{thm:HS.explicit}
yields explicit formulae for the functions $R_{n,p}$ and,
specifically,
$$Z^{\mathsf{C}}_{n,\ell,p}(s) =(1-p^{\ell-s})^{-1}(1-p^{\binom{n+1}{2}-s})^{-1} \sum_{I\subseteq [n-1]}W_{n,I}\left( p, \left( p^{\binom{i+1}{2}+\ell-ns}\right)_{i=1}^n, p^{-\ell}\right).$$

In the next theorem we show that the primitive local Satake generating
function can be viewed as a ``$p$-analogue'' of the fine Hilbert series
of a Stanley--Reisner ring. Let $V$ be a finite set. If
$\Delta\subseteq 2^V$ is a simplicial complex on~$V$, then the
Stanley--Reisner ring of $\Delta$ over a ring $K$ is
$$
  K[\Delta] = K[X_v ~|~ v\in V] / (\prod_{v\in \sigma} X_v ~|~ \sigma\in 2^V \setminus \Delta).  
$$

\begin{thmabc}\label{thmmain:SR-ring}
  For all $n\in\N$, let $\Delta_n$ be the $n$-simplex with vertices $[n]$ and
  $\Delta = \mathrm{sd}(\partial\Delta_n)$, the barycentric subdivision of
  boundary of $\Delta_n$, with vertices given by the nonempty subsets of $[n]$.
  Let $\bm{y} = (y_I : \varnothing \neq I\subseteq [n])$ and $\phi:\Z\llbracket
  \bm{y}\rrbracket \rightarrow \Z\llbracket \bm{x}, X\rrbracket$ via $y_I \mapsto x_0X\prod_{i\in I} x_i$. Then 
  \begin{align*} 
    \left.\Satakegen_{n,p}^{\prim}(\bm{x}, X)\right|_{p\to 1} &= \phi(\mathrm{Hilb}(\Z[\Delta]; \bm{y})) = \sum_{\sigma\in\Delta} \prod_{J\in \sigma} \dfrac{\phi(y_J)}{1-\phi(y_J)} . 
  \end{align*} 
\end{thmabc}

With \Cref{thmmain:SR-ring}, we come full circle and relate the local
Satake generating function $\Satakegen_{n,p}$ to the Ehrhart series of
the $n$-cube.

\begin{cor}\label{conj:wurfel}
  For all $n\in\N$, let $P$ be the $n$-cube. Then 
  \begin{align*} 
    \left.\Satakegen_{n,p}(\bm{1}, X)\right|_{p\to 1} &= \mathrm{Ehr}_P(X) = \dfrac{\mathrm{E}_n(X)}{(1-X)^{n+1}},
  \end{align*}
  where $\mathrm{E}_n(X)=\sum_{\sigma\in S_n}X^{\mathrm{des}(\sigma)}$ is the
  Eulerian polynomial. 
\end{cor}

\begin{proof} 
  It follows from \Cref{thmmain:SR-ring} that
  \begin{align} \label{eqn:reduced}
    (1-X)^2 \left.\Satakegen_{n,p}(\bm{1}, X)\right|_{p\to 1} &= \sum_{\sigma\in\Delta} \prod_{J\in \sigma} \dfrac{X}{1-X},
  \end{align} 
  where $\Delta$ is the barycentric subdivision of the boundary of the
  $n$-simplex. From~\cite[Theore.~9.1]{Petersen/15} and
  Equation~\eqref{eqn:reduced} it follows that
  \begin{align*} 
    \left.\Satakegen_{n,p}(\bm{1}, X)\right|_{p\to 1} &= \dfrac{\mathrm{E}_n(X)}{(1-X)^{n+1}} = \sum_{k\geq 0} (k+1)^nX^k = \mathrm{Ehr}_P(X).\qedhere
  \end{align*} 
\end{proof}

\subsection{The type--A story}
\label{sec:type-A}

Our work was inspired by Gunnells and Rodriguez
Villegas. In~\cite{GRV/07} they considered type-$\mathsf{A}$ versions
of some of the questions outlined above. We paraphrase parts
of~\cite{GRV/07} from the perspective of our work in type~$\mathsf{C}$.
For a prime $p$ we define the (\emph{local}) \emph{Ehrhart--Hecke zeta
function} (of type $\mathsf{A}$) as
\begin{equation}\label{def:Z.local}
  \mathcal{Z}^{\mathsf{A}}_{n,\ell,p}(s) = \sum_{\substack{\Z_p^n\leq \Lambda \leq \Q_p^n \\ |\Lambda : \Z_p^n| <\infty}}\frac{c_\ell^{\Lambda}(P)}{c_{\ell}(P)}|\Lambda:\Zp^n|^{-s}.
\end{equation} 

Let $\Gamma_{n}^{\mathsf{A}} = \GL_n(\Z)$ and $G_n^{\mathsf{A}} =
\Mat_n(\Z)\cap\GL_n(\Q)$. For $m\in\N$, let
\begin{align*} 
  D_n^{\mathsf{A}}(m) &= \{ g\in G_n^{\mathsf{A}} ~|~ |\det(g)| = m \},
\end{align*}
so $D_n^{\mathsf{A}}(m)$ is a finite union of double cosets relative to $\Gamma_{n}^{\mathsf{A}}$. We define 
\begin{align*} 
  T_n^{\mathsf{A}}(m) &= \sum_{g\in \Gamma_{n}^{\mathsf{A}}\setminus D_n^{\mathsf{A}}(m)/\Gamma_{n}^{\mathsf{A}}} \Gamma_{n}^{\mathsf{A}} g \Gamma_{n}^{\mathsf{A}} ,
\end{align*} 
where the sum runs over a set of representatives of the double cosets,
which is an element of the Hecke algebra determined by
$(\Gamma_n^{\mathsf{A}},G_n^{\mathsf{A}})$. Moreover, if
$\gcd(m,m')=1$, then
\[ 
  T_n^{\mathsf{A}}(m)T_n^{\mathsf{A}}(m') = T_n^{\mathsf{A}}(mm'). 
\]
For $k\in [n]_0$ define $\pi_k(p) = \diag(1, \dots, 1, \overbrace{p, \dots, p}^k)$ and  
$
  T_n^{\mathsf{A}}(p, k) =  \Gamma_{n}^{\mathsf{A}}\pi_{k}(p)\Gamma_{n}^{\mathsf{A}},
$ 
which decomposes into a finite (disjoint) union of right cosets relative to
$\Gamma_{n}^{\mathsf{A}}$. 

Gunnells and Rodriguez Villegas~\cite{GRV/07}
considered the following action of the Hecke algebra on the Ehrhart polynomial
$E(P)=E^{\Lambda_0}(P)$ of $P$:
\begin{align}\label{eqn:Hecke-action}
  T_n^{\mathsf{A}}(p, k) E(P) &= \sum_{g\in \Gamma_{n}^{\mathsf{A}}\pi_{k}(p)\Gamma_{n}^{\mathsf{A}} / \Gamma_{n}^{\mathsf{A}}} E(g\cdot P) ,
\end{align} 
where the sum runs over a set of right coset representatives. The
action in~\eqref{eqn:Hecke-action} is independent of the chosen
representatives since $\Gamma_n^{\mathsf{A}}$ comprises bijections of
$\Z^n$. Our definition in~\eqref{eqn:Hecke-action} differs
from~\cite{GRV/07} only cosmetically via~\eqref{eqn:Ehrhart-id}.

Denote by $\mathrm{Gr}(\ell, n, p)$ the set of $\ell$-dimensional
subspaces in $\Fp^n$. For $n\in\N$, $\ell, k\in [n]_0$, and $U\in
\mathrm{Gr}(\ell, n, p)$, define
\begin{align*} 
  \nu_{n, k, \ell}^{\mathsf{A}}(p) &= \sum_{W\in \mathrm{Gr}(k, n,
    p)}\# (U\cap W).
\end{align*} 
Let $\psi_{n,\ell}^{\mathsf{A}} : \Q[x_1^{\pm 1}, \dots, x_n^{\pm
    1}]\to \Q$ be given by $x_n\mapsto p^{\ell}$ and $x_i\mapsto p^i$
for all $i\in [n-1]$. Let further $\omega$ denote the Satake
isomorphism from the $p$-primary part of the Hecke algebra associated
with $(\Gamma_n^{\mathsf{A}}, G_n^{\mathsf{A}})$, written
$\mathcal{H}_p^{\mathsf{A}}$, to the symmetric subring of $\Q[x_1^{\pm
    1}, \dots, x_n^{\pm 1}]$. 

Let $s_{n,k}(x_1, \dots, x_{n})$ be the (homogeneous) elementary
symmetric polynomial of degree $k$, and set $s_{n, -1}=0$.

\begin{thm}[{\cite{GRV/07}}]\label{thm:A}
  For $n\in\N$, $k,\ell\in [n]_0$, and a prime $p$, we have
  \begin{align*} 
    \nu_{n,k,\ell}^{\mathsf{A}}(p) = p^k\binom{n-1}{k}_p + p^{\ell}\binom{n-1}{k-1}_p = \psi_{n,\ell}^{\mathsf{A}}(\omega(T_n^{\mathsf{A}}(p,k))). 
  \end{align*}
  Moreover, 
  \[ 
    \mathcal{Z}_{n,\ell,p}^{\mathsf{A}}(s) = (1-p^{\ell-s})^{-1} \prod_{k=1}^{n-1} (1-p^{k-s})^{-1}. 
  \]
\end{thm}

\begin{proof} 
  First we prove the claims concerning
  $\nu_{n,k,\ell}^{\mathsf{A}}(p)$. Therefore,
  \begin{align*} 
    \nu_{n,k,\ell}^{\mathsf{A}}(p) &= \binom{n}{k}_p - \binom{n-1}{k-1}_p + p^{\ell}\binom{n-1}{k-1}_p & (\text{\cite[Lem.~3.3]{GRV/07}}) \\
    &= p^k \binom{n-1}{k}_p + p^{\ell} \binom{n-1}{k-1}_p & (\text{Pascal identity}) \\
    &= p^k s_{n-1,k}(1, p, \dots, p^{n-2}) + p^{\ell} s_{n-1,k-1}(1, p, \dots, p^{n-2}) &  (\text{\cite[Ex.~I.2.3]{Macdonald/95}}) \\
    &= p^{-\binom{k}{2}}\psi_{n,\ell}^{\mathsf{A}}(s_{n,k}) \\ 
    &= \psi_{n,\ell}^{\mathsf{A}}(\omega(T_n^{\mathsf{A}}(p,k))). & (\text{\cite[Lem.~3.2.21]{Andrianov/87}})
  \end{align*} 

  We now tend to the last claim.  Tamagawa~\cite{Tamagawa/63}
  established the identity
  \begin{equation}\label{eqn:tamagawa}
    \sum_{m\geq 0} T_n^{\mathsf{A}}(p^m) X^m = \left( \sum_{k=0}^n (-1)^k p^{\binom{k}{2}}T_{n}^{\mathsf{A}}(p,k)X^k\right)^{-1} \in \mathcal{H}_p^{\mathsf{A}} \llbracket X \rrbracket .
  \end{equation}
  Applying $\psi_{n,\ell}^{\mathsf{A}}\circ \omega$ to~\eqref{eqn:tamagawa} and
  setting $X=p^{-s}$, we have 
  \begin{equation*}
    \sum_{m\geq 0} \psi_{n,\ell}^{\mathsf{A}}(\omega(T_n^{\mathsf{A}}(p^m))) p^{-ms} = \left( \sum_{k=0}^n \psi_{n,\ell}^{\mathsf{A}}(s_{n,k})(-p)^{-ks}\right)^{-1} = (1-p^{\ell-s})^{-1} \prod_{k=1}^{n-1} (1-p^{k-s})^{-1} .
  \end{equation*}
  Since $\nu_{n,k,\ell}^{\mathsf{A}}(p)$ is an eigenvalue for $T_n(p,k)$, it
  follows that 
  \[ 
    \mathcal{Z}^{\mathsf{A}}_{n,\ell, p}(s) = \sum_{m\geq 0} \psi_{n,\ell}^{\mathsf{A}}(\omega(T_n^{\mathsf{A}}(p^m))) p^{-ms}. \qedhere 
  \]
\end{proof}

%
\begin{cor}
  Let $\zeta(s)$ be the Riemann zeta function. For $n\in\N$ and $\ell\in [n]_0$,
  we have
  \[ 
  \prod_{\text{prime }p} \mathcal{Z}_{n, \ell, p}^{\mathsf{A}}(s) = \zeta(s-\ell) \prod_{k=1}^{n-1} \zeta(s-k) .
  \] 
\end{cor}

\section{Examples}

\subsection{Hecke eigenfunctions}

We give some explicit examples, showing in \Cref{fig:eigenfunctions} that the
eigenfunctions of \Cref{thmmain:Ehrhart-Satake} depend significantly on the
polytope. We do this by displaying a graph whose vertices correspond to
homothety classes of lattices. We evaluate the functions
$\mathscr{E}_{n,p\ell,P}$ on $\Lambda_{\mathrm{m}}$ for each homothety class
$[\Lambda]$. 

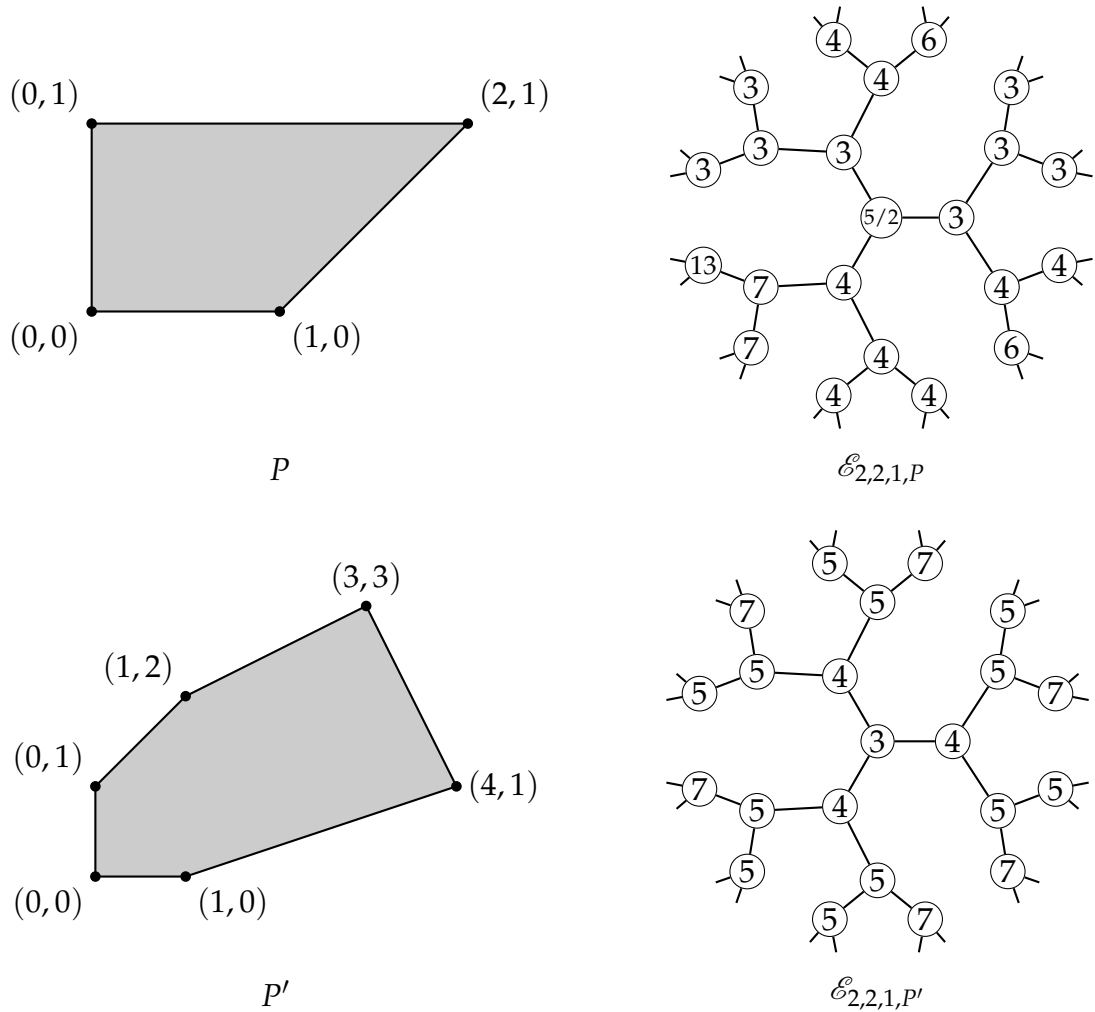
\begin{figure}[hb!]
  \centering 
  \begin{subfigure}[b]{\textwidth}
    \centering 
    \begin{tikzpicture}
      \pgfmathsetmacro{\unit}{2.5}
      \node at (-4,0) {
        \begin{tikzpicture}
          \coordinate (A) at (0,0);
          \coordinate (B) at (\unit,0);
          \coordinate (C) at (0,\unit);
          \coordinate (D) at (2*\unit,\unit);
        
          \draw[thick, fill=black!20!white] (A) -- (B) -- (D) -- (C) -- cycle;
        
          \foreach \point in {A, B, C, D}
            \fill[black] (\point) circle (2pt);

          \node[below left] at (A) {$(0,0)$};
          \node[below right] at (B) {$(1,0)$};
          \node[above left] at (C) {$(0,1)$};
          \node[above right] at (D) {$(2,1)$};
        \end{tikzpicture}
      };
      \node at (4,0) {
        \begin{tikzpicture}
          \pgfmathsetmacro{\r}{1}
          \pgfmathsetmacro{\rr}{1.85}
          \pgfmathsetmacro{\rrr}{2.45}
          \pgfmathsetmacro{\rrrr}{3}
          \node[circle, draw, fill=white, inner sep=0.3pt] (0) at (0,0) {{\scriptsize $5/2$}};
          \node[circle, draw, fill=white, inner sep=1pt] (a) at (120:\r) {$3$};
          \node[circle, draw, fill=white, inner sep=1pt] (b) at (240:\r) {$4$};
          \node[circle, draw, fill=white, inner sep=1pt] (c) at (360:\r) {$3$};
          \node[circle, draw, fill=white, inner sep=1pt] (a1) at (90:\rr) {$4$};
          \node[circle, draw, fill=white, inner sep=1pt] (a2) at (150:\rr) {$3$};
          \node[circle, draw, fill=white, inner sep=1pt] (b1) at (210:\rr) {$7$};
          \node[circle, draw, fill=white, inner sep=1pt] (b2) at (270:\rr) {$4$};
          \node[circle, draw, fill=white, inner sep=1pt] (c1) at (330:\rr) {$4$};
          \node[circle, draw, fill=white, inner sep=1pt] (c2) at (390:\rr) {$3$};
          \node[circle, draw, fill=white, inner sep=1pt] (a1i) at (75:\rrr) {$6$};
          \node[circle, draw, fill=white, inner sep=1pt] (a1ii) at (105:\rrr) {$4$};
          \node[circle, draw, fill=white, inner sep=1pt] (a2i) at (135:\rrr) {$3$};
          \node[circle, draw, fill=white, inner sep=1pt] (a2ii) at (165:\rrr) {$3$};
          \node[circle, draw, fill=white, inner sep=0.5pt] (b1i) at (195:\rrr) {{\footnotesize $13$}};
          \node[circle, draw, fill=white, inner sep=1pt] (b1ii) at (225:\rrr) {$7$};
          \node[circle, draw, fill=white, inner sep=1pt] (b2i) at (255:\rrr) {$4$};
          \node[circle, draw, fill=white, inner sep=1pt] (b2ii) at (285:\rrr) {$4$};
          \node[circle, draw, fill=white, inner sep=1pt] (c1i) at (315:\rrr) {$6$};
          \node[circle, draw, fill=white, inner sep=1pt] (c1ii) at (345:\rrr) {$4$};
          \node[circle, draw, fill=white, inner sep=1pt] (c2i) at (375:\rrr) {$3$};
          \node[circle, draw, fill=white, inner sep=1pt] (c2ii) at (405:\rrr) {$3$};
          \node (a1ir) at (70:\rrrr) {};
          \node (a1il) at (80:\rrrr) {};
          \node (a1iir) at (100:\rrrr) {};
          \node (a1iil) at (110:\rrrr) {};
          \node (a2ir) at (130:\rrrr) {};
          \node (a2il) at (140:\rrrr) {};
          \node (a2iir) at (160:\rrrr) {};
          \node (a2iil) at (170:\rrrr) {};
          \node (b1ir) at (190:\rrrr) {};
          \node (b1il) at (200:\rrrr) {};
          \node (b1iir) at (220:\rrrr) {};
          \node (b1iil) at (230:\rrrr) {};
          \node (b2ir) at (250:\rrrr) {};
          \node (b2il) at (260:\rrrr) {};
          \node (b2iir) at (280:\rrrr) {};
          \node (b2iil) at (290:\rrrr) {};
          \node (c1ir) at (310:\rrrr) {};
          \node (c1il) at (320:\rrrr) {};
          \node (c1iir) at (340:\rrrr) {};
          \node (c1iil) at (350:\rrrr) {};
          \node (c2ir) at (370:\rrrr) {};
          \node (c2il) at (380:\rrrr) {};
          \node (c2iir) at (400:\rrrr) {};
          \node (c2iil) at (410:\rrrr) {};
  
          \draw[thick] (0) -- (a);
          \draw[thick] (0) -- (b);
          \draw[thick] (0) -- (c);
          \draw[thick] (a) -- (a1);
          \draw[thick] (a) -- (a2);
          \draw[thick] (b) -- (b1);
          \draw[thick] (b) -- (b2);
          \draw[thick] (c) -- (c1);
          \draw[thick] (c) -- (c2);
          \draw[thick] (a1) -- (a1i);
          \draw[thick] (a1) -- (a1ii);
          \draw[thick] (a2) -- (a2i);
          \draw[thick] (a2) -- (a2ii);
          \draw[thick] (b1) -- (b1i);
          \draw[thick] (b1) -- (b1ii);
          \draw[thick] (b2) -- (b2i);
          \draw[thick] (b2) -- (b2ii);
          \draw[thick] (c1) -- (c1i);
          \draw[thick] (c1) -- (c1ii);
          \draw[thick] (c2) -- (c2i);
          \draw[thick] (c2) -- (c2ii);
          \draw[thick] (a1i) -- (a1ir);
          \draw[thick] (a1i) -- (a1il);
          \draw[thick] (a1ii) -- (a1iir);
          \draw[thick] (a1ii) -- (a1iil);
          \draw[thick] (a2i) -- (a2ir);
          \draw[thick] (a2i) -- (a2il);
          \draw[thick] (a2ii) -- (a2iir);
          \draw[thick] (a2ii) -- (a2iil);
          \draw[thick] (b1i) -- (b1ir);
          \draw[thick] (b1i) -- (b1il);
          \draw[thick] (b1ii) -- (b1iir);
          \draw[thick] (b1ii) -- (b1iil);
          \draw[thick] (b2i) -- (b2ir);
          \draw[thick] (b2i) -- (b2il);
          \draw[thick] (b2ii) -- (b2iir);
          \draw[thick] (b2ii) -- (b2iil);
          \draw[thick] (c1i) -- (c1ir);
          \draw[thick] (c1i) -- (c1il);
          \draw[thick] (c1ii) -- (c1iir);
          \draw[thick] (c1ii) -- (c1iil);
          \draw[thick] (c2i) -- (c2ir);
          \draw[thick] (c2i) -- (c2il);
          \draw[thick] (c2ii) -- (c2iir);
          \draw[thick] (c2ii) -- (c2iil);
        \end{tikzpicture}
      };
      \node at (-4,-3.35) {$P$};
      \node at (4,-3.35) {$\mathscr{E}_{2,2,1,P}$};
    \end{tikzpicture}
  \end{subfigure}
  \begin{subfigure}[b]{\textwidth}
    \centering 
    \begin{tikzpicture}
      \node at (-4,0) {
        \begin{tikzpicture}
          \pgfmathsetmacro{\unit}{1.2}
          \coordinate (A) at (0,0);
          \coordinate (B) at (\unit,0);
          \coordinate (C) at (0,\unit);
          \coordinate (D) at (\unit,2*\unit);
          \coordinate (E) at (3*\unit,3*\unit);
          \coordinate (F) at (4*\unit,\unit);
        
          \draw[thick, fill=black!20!white] (A) -- (B) -- (F) -- (E) -- (D) -- (C) -- cycle;

          \foreach \point in {A, B, C, D, E, F}
            \fill[black] (\point) circle (2pt);
        
          \node[below left] at (A) {$(0,0)$};
          \node[below right] at (B) {$(1,0)$};
          \node[above left] at (C) {$(0,1)$};
          \node[above left] at (D) {$(1,2)$};
          \node[above] at (E) {$(3,3)$};
          \node[right] at (F) {$(4,1)$};
        \end{tikzpicture}
      };
      \node at (4,0) {
        \begin{tikzpicture}
          \pgfmathsetmacro{\r}{1}
          \pgfmathsetmacro{\rr}{1.85}
          \pgfmathsetmacro{\rrr}{2.45}
          \pgfmathsetmacro{\rrrr}{3}
          \node[circle, draw, fill=white, inner sep=0.8pt] (0) at (0,0) {$3$};
          \node[circle, draw, fill=white, inner sep=1pt] (a) at (120:\r) {$4$};
          \node[circle, draw, fill=white, inner sep=1pt] (b) at (240:\r) {$4$};
          \node[circle, draw, fill=white, inner sep=1pt] (c) at (360:\r) {$4$};
          \node[circle, draw, fill=white, inner sep=1pt] (a1) at (90:\rr) {$5$};
          \node[circle, draw, fill=white, inner sep=1pt] (a2) at (150:\rr) {$5$};
          \node[circle, draw, fill=white, inner sep=1pt] (b1) at (210:\rr) {$5$};
          \node[circle, draw, fill=white, inner sep=1pt] (b2) at (270:\rr) {$5$};
          \node[circle, draw, fill=white, inner sep=1pt] (c1) at (330:\rr) {$5$};
          \node[circle, draw, fill=white, inner sep=1pt] (c2) at (390:\rr) {$5$};
          \node[circle, draw, fill=white, inner sep=1pt] (a1i) at (75:\rrr) {$7$};
          \node[circle, draw, fill=white, inner sep=1pt] (a1ii) at (105:\rrr) {$5$};
          \node[circle, draw, fill=white, inner sep=1pt] (a2i) at (135:\rrr) {$7$};
          \node[circle, draw, fill=white, inner sep=1pt] (a2ii) at (165:\rrr) {$5$};
          \node[circle, draw, fill=white, inner sep=1pt] (b1i) at (195:\rrr) {$7$};
          \node[circle, draw, fill=white, inner sep=1pt] (b1ii) at (225:\rrr) {$5$};
          \node[circle, draw, fill=white, inner sep=1pt] (b2i) at (255:\rrr) {$5$};
          \node[circle, draw, fill=white, inner sep=1pt] (b2ii) at (285:\rrr) {$7$};
          \node[circle, draw, fill=white, inner sep=1pt] (c1i) at (315:\rrr) {$7$};
          \node[circle, draw, fill=white, inner sep=1pt] (c1ii) at (345:\rrr) {$5$};
          \node[circle, draw, fill=white, inner sep=1pt] (c2i) at (375:\rrr) {$7$};
          \node[circle, draw, fill=white, inner sep=1pt] (c2ii) at (405:\rrr) {$5$};
          \node (a1ir) at (70:\rrrr) {};
          \node (a1il) at (80:\rrrr) {};
          \node (a1iir) at (100:\rrrr) {};
          \node (a1iil) at (110:\rrrr) {};
          \node (a2ir) at (130:\rrrr) {};
          \node (a2il) at (140:\rrrr) {};
          \node (a2iir) at (160:\rrrr) {};
          \node (a2iil) at (170:\rrrr) {};
          \node (b1ir) at (190:\rrrr) {};
          \node (b1il) at (200:\rrrr) {};
          \node (b1iir) at (220:\rrrr) {};
          \node (b1iil) at (230:\rrrr) {};
          \node (b2ir) at (250:\rrrr) {};
          \node (b2il) at (260:\rrrr) {};
          \node (b2iir) at (280:\rrrr) {};
          \node (b2iil) at (290:\rrrr) {};
          \node (c1ir) at (310:\rrrr) {};
          \node (c1il) at (320:\rrrr) {};
          \node (c1iir) at (340:\rrrr) {};
          \node (c1iil) at (350:\rrrr) {};
          \node (c2ir) at (370:\rrrr) {};
          \node (c2il) at (380:\rrrr) {};
          \node (c2iir) at (400:\rrrr) {};
          \node (c2iil) at (410:\rrrr) {};
  
          \draw[thick] (0) -- (a);
          \draw[thick] (0) -- (b);
          \draw[thick] (0) -- (c);
          \draw[thick] (a) -- (a1);
          \draw[thick] (a) -- (a2);
          \draw[thick] (b) -- (b1);
          \draw[thick] (b) -- (b2);
          \draw[thick] (c) -- (c1);
          \draw[thick] (c) -- (c2);
          \draw[thick] (a1) -- (a1i);
          \draw[thick] (a1) -- (a1ii);
          \draw[thick] (a2) -- (a2i);
          \draw[thick] (a2) -- (a2ii);
          \draw[thick] (b1) -- (b1i);
          \draw[thick] (b1) -- (b1ii);
          \draw[thick] (b2) -- (b2i);
          \draw[thick] (b2) -- (b2ii);
          \draw[thick] (c1) -- (c1i);
          \draw[thick] (c1) -- (c1ii);
          \draw[thick] (c2) -- (c2i);
          \draw[thick] (c2) -- (c2ii);
          \draw[thick] (a1i) -- (a1ir);
          \draw[thick] (a1i) -- (a1il);
          \draw[thick] (a1ii) -- (a1iir);
          \draw[thick] (a1ii) -- (a1iil);
          \draw[thick] (a2i) -- (a2ir);
          \draw[thick] (a2i) -- (a2il);
          \draw[thick] (a2ii) -- (a2iir);
          \draw[thick] (a2ii) -- (a2iil);
          \draw[thick] (b1i) -- (b1ir);
          \draw[thick] (b1i) -- (b1il);
          \draw[thick] (b1ii) -- (b1iir);
          \draw[thick] (b1ii) -- (b1iil);
          \draw[thick] (b2i) -- (b2ir);
          \draw[thick] (b2i) -- (b2il);
          \draw[thick] (b2ii) -- (b2iir);
          \draw[thick] (b2ii) -- (b2iil);
          \draw[thick] (c1i) -- (c1ir);
          \draw[thick] (c1i) -- (c1il);
          \draw[thick] (c1ii) -- (c1iir);
          \draw[thick] (c1ii) -- (c1iil);
          \draw[thick] (c2i) -- (c2ir);
          \draw[thick] (c2i) -- (c2il);
          \draw[thick] (c2ii) -- (c2iir);
          \draw[thick] (c2ii) -- (c2iil);
        \end{tikzpicture}
      };
      \node at (-4,-3.35) {$P'$};
      \node at (4,-3.35) {$\mathscr{E}_{2,2,1,P'}$};
    \end{tikzpicture}
  \end{subfigure}
  \caption{Polytopes and some values of $\mathscr{E}_{2,2,1,P}$ displayed on
  lattices in the affine building of type $\widetilde{\mathsf{A}}_1$ associated
  with the group $\GSp_2(\Q_p) \cong \GL_2(\Q_p)$. The center vertex corresponds
  to the homothety class of the identity, and the values are the linear
  coefficients of the Ehrhart polynomials with respect to the corresponding
  lattices.}
  \label{fig:eigenfunctions}
\end{figure}

\subsection{Local Ehrhart--Hecke zeta functions}

For $n\in[3]$ and $\ell\in[2n]_0$, we record the rational functions
$W_{n,\ell}(X,Y)\in\Q(X,Y)$ where, for all primes, $\mathcal{Z}_{n, \ell,
p}^{\mathsf{C}}(s) = W_{n,\ell}(p, p^{-ns})$. We computed these with
SageMath~\cite{sagemath}.
\begin{align*}
  W_{1,\ell}(X,Y) &= \dfrac{1}{(1 - XY) (1 - X^{\ell}Y)} \\ 
  W_{2,\ell}(X,Y) &= \dfrac{1-X^{2+\ell}Y^2}{(1 -
  X^2Y) (1 - X^3Y) (1 - X^{\ell}Y) (1 - X^{\ell+1}Y)} \\
  W_{3, \ell}(X, Y) &= \dfrac{1 + (X^{1+\ell} + X^4)Y - A_{\ell}(X)Y^2 + (X^{6+2\ell} + X^{9+\ell})Y^3 + X^{10+2\ell}Y^4}{(1 - X^3Y) (1 - X^5Y) (1 - X^6Y) (1 - X^{\ell}Y) (1 - X^{2+\ell}Y) (1 - X^{3+\ell}Y)} \\
  W_{4, \ell}(X, Y) &= \dfrac{N_{4, \ell}(X, Y)}{D_{4, \ell}(X, Y)},
\end{align*}
where $A_{\ell}(X) = X^{7+\ell} + 2X^{6+\ell} + 2X^{4+\ell} + X^{3+\ell}$,
\begin{align*} 
  N_{4, \ell}(X, Y) &= 1 + (X^5+X^6+X^7+X^8+X^{1+\ell}+X^{2+\ell}+X^{3+\ell}+X^{4+\ell})Y + (X^{13} \\
  &\qquad - X^{4+\ell} - 2X^{5+\ell} - 2X^{6+\ell} - 2X^{7+\ell} - 2X^{8+\ell} - 2X^{9+\ell} - 3X^{10+\ell} \\
  &\qquad - 2X^{11+\ell} - 2X^{12+\ell} - 2X^{13+\ell} - X^{14+\ell} + X^{5 + 2\ell})Y^2 + (X^{14 + \ell} \\
  &\qquad - X^{18 + \ell} + X^{10+2\ell} - X^{14+2\ell})Y^3 - (X^{23+\ell} - X^{14+2\ell} - 2X^{15+2\ell} \\
  &\qquad - 2X^{16+2\ell} - 2X^{17+2\ell} - 3X^{18+2\ell} - 2X^{19+2\ell} - 2X^{20+2\ell} - 2X^{21+2\ell} \\
  &\qquad - 2X^{22+2\ell} - 2X^{23+2\ell} - X^{24+2\ell} + X^{15+3\ell})Y^4 - (X^{24+2\ell} + X^{25+2\ell} \\
  &\qquad + X^{26+2\ell} + X^{27+2\ell} + X^{20+3\ell} + X^{21+3\ell} + X^{22+3\ell} + X^{23+3\ell})Y^5 \\
  &\qquad - X^{28+3\ell}Y^6,  \\ 
  D_{4,\ell}(X, Y) &= (1-X^4Y) (1-X^7Y) (1-X^9Y) (1-X^{10}Y) \\
  &\qquad \times (1-X^{\ell}Y) (1-X^{3+\ell}Y) (1-X^{5+\ell}Y) (1-X^{6+\ell}Y) .
\end{align*}

\printbibliography

\end{document}